\documentclass[amstex]{amsart}
\usepackage{mathrsfs}
\usepackage{amssymb}
\usepackage{amsfonts}
\usepackage{amsfonts,amsmath,amsthm, amssymb}
\usepackage{array}
\usepackage{latexsym, euscript, epic, eepic}
\usepackage{tikz}
\usetikzlibrary{matrix}

\newtheorem{theorem}{Theorem}[section]
\newtheorem{lemma}[theorem]{Lemma}
\newtheorem{proposition}[theorem]{Proposition}
\newtheorem{corollary}[theorem]{Corollary}
\newtheorem{remark}[theorem]{Remark}
\newtheorem{definition}[theorem]{Definition}

\begin{document}

\title[]{Teich\-m\"ul\-ler spaces of generalized symmetric homeomorphisms} 

\author[H. Wei]{Huaying Wei} 
\address{Department of Mathematics and Statistics, Jiangsu Normal University \endgraf Xuzhou 221116, PR China} 
\email{hywei@jsnu.edu.cn} 

\author[K. Matsuzaki]{Katsuhiko Matsuzaki}
\address{Department of Mathematics, School of Education, Waseda University \endgraf
Shinjuku, Tokyo 169-8050, Japan}
\email{matsuzak@waseda.jp}

\thanks{Research supported by the National Natural Science Foundation of China (Grant No. 11501259)
and Japan Society for the Promotion of Science (KAKENHI 18H01125).}

\begin{abstract}
We introduce the concept of a new kind of symmetric homeomorphisms on the unit circle, which is derived from the generalization of 
symmetric homeomorphisms on the real line. By the investigation of the barycentric extension 
for this class of circle homeomorphisms and the biholomorphic automorphisms induced by trivial Beltrami coefficients, 
we endow a complex Banach manifold structure on the space of those generalized symmetric homeomorphisms.  
\end{abstract}

\maketitle

\section{Introduction}
The universal Teich\-m\"ul\-ler space plays a fundamental role in the quasiconformal theory of Teich\-m\"ul\-ler spaces, and it is also an important object in mathematical physics. 
The universal Teich\-m\"ul\-ler space $T$ 
can be defined as the group $\rm QS$ of all quasisymmetric homeomorphisms of the unit circle 
$\mathbb{S} = \{z\in\mathbb{C} \mid  |z|=1\}$ modulo the left action of the group 
$\mbox{\rm M\"ob}(\mathbb{S})$ of all M\"obius transformations of $\mathbb{S}$, i.e., 
$T= \mbox{\rm M\"ob}(\mathbb{S}) \backslash \rm QS$. It can also be defined on the real line $\mathbb{R}$ by the conjugation of  
a M\"obius transformation. 

Several subclasses of quasisymmetric homeomorphisms and their Teich\-m\"ul\-ler spaces, spread in different directions,  were introduced and studied for various purposes in the literature. We refer to the books \cite{Ah66, GL, Le,Na88} and the papers \cite{AZ,Cu,GS,Ma,Se,SW,TT} for  introducing the subject matters in more details. Our work in this paper is mainly based on the subclass consisting of all symmetric homeomorphisms on the real line $\mathbb{R}$, and is motivated by recent work of Hu, Wu, and Shen \cite{HWS}.

A symmetric homeomorphism on the real line $\mathbb{R}$ was first studied in \cite{Ca} when Carleson discussed absolute continuity of  quasisymmetric homeomorphisms. It was proved that $h$ is symmetric if and only if $h$ can be extended to an asymptotically conformal homeo\-morphism $f$ of the upper half-plane $\mathbb{U}$ onto itself. Later, Gardiner and Sullivan \cite{GS} introduced the
concept of the symmetric structure on $\mathbb S$ by relying on this relationship.
By an asymptotically conformal homeomorphism $f$ of the upper half-plane $\mathbb{U}$, we mean that its complex dilatation 
$\mu = \Bar{\partial}f/\partial f$ satisfies that
$$
{\rm ess}\!\!\!\!\sup_{y \leq t\qquad}\!\!\!\!|\mu(x+iy)| \to 0 \quad (t \to 0).
$$
In fact, the Beurling--Ahlfors extension of $h$ is asymptotically conformal when $h$ is symmetric (see \cite{Ca,GS,Ma}). 
Based on these results, Hu, Wu, and Shen \cite{HWS} endowed the symmetric Teich\-m\"ul\-ler space $T_*(\mathbb R)$
on the real line with a complex Banach manifold structure. 
Namely, $T_*(\mathbb R)$ can be embedded as a bounded domain in a certain Banach space.

The conjugation by the Cayley transformation $\phi_{\xi}(z) = \xi(z-i)/(z+i)$ 
that maps $\mathbb{R}$ onto $\mathbb{S}$ with $\infty \mapsto \xi$ can transfer a symmetric homeomorphism $h$ of $\mathbb{R}$ to $\hat{h} = \phi_{\xi}\circ h\circ (\phi_{\xi})^{-1}$ of $\mathbb{S}$. Then, we see that
$\hat{h}$ extends to a quasiconformal self-homeomorphism of the unit disk $\mathbb{D}$ whose complex dilatation 
can be arbitrarily small outside some horoball tangent at $\xi$ in $\mathbb{D}$. 
In this paper, we generalize this transformation for only one tangent point to the case of plural tangent points by specifying 
a set $X$ of points $\xi \in \mathbb{S}$. We define the set ${\rm QS}^X_*$ of the generalized symmetric homeomorphisms for 
$X \subset \mathbb{S}$ by the boundary extension of quasiconformal self-homeomorphisms of $\mathbb{D}$
whose complex dilatations are arbitrarily small outside some horoballs tangent at all $\xi \in X$. 
We denote the set of such complex dilatations by $M^X_*(\mathbb D)$.
Then, we introduce 
the generalized symmetric Teich\-m\"ul\-ler space $T^X_*$ by ${\rm QS}^X_*$ modulo $\mbox{\rm M\"ob}(\mathbb{S})$
as well as the Teich\-m\"ul\-ler projection $\pi:M^X_*(\mathbb D) \to T^X_*$.

The Teich\-m\"ul\-ler space $T^X_*$ is of interest because it lies between the universal Teich\-m\"ul\-ler space $T$ and 
its little subspace $T_0=\mbox{\rm M\"ob}(\mathbb{S}) \backslash {\rm Sym}$ made of the symmetric homeomorphisms on $\mathbb S$.
We expect that a family $T^{X_n}_*$ defined by a sequence of increasing subsets $X_n \subset \mathbb S$ 
can give an interpolation between $T_0$ and $T$.

To investigate the structure of $T^X_*$, we show that
the barycentric extension due to Douady and Earle \cite{DE} gives a proper quasiconformal extension 
for the elements of ${\rm QS}^X_*$ to those whose complex dilatations are in $M^X_*(\mathbb D)$. 
The proof is given by an adaptation of the argument by Earle, Markovic, and Saric \cite{EMS}.
After this, we can endow $T_*^X$ with a complex Banach manifold structure as a bounded domain in the corresponding Banach space 
$B^X_*(\mathbb D^*)$. 
The proof is carried out by using the Bers Schwarzian derivative map $\Phi:M^X_*(\mathbb D) \to B^X_*(\mathbb D^*)$
and showing that it is a holomorphic split submersion as usual.
At this stage, the barycentric extension is also useful.
Moreover, we point out that due to lack of a group structure on $M^X_*(\mathbb D)$, we need some careful arguments
for holomorphic split submersion differently from the usual case.


\section{Preliminaries}
In this section, we review basic facts on the universal Teich\-m\"ul\-ler space
and the symmetric Teich\-m\"ul\-ler space on the real line.

\subsection{Universal Teich\-m\"ul\-ler space}
We begin with a standard theory of the universal Teich\-m\"ul\-ler space. For details, we can refer to monographs
\cite{GL, Le, Na88}.
The universal Teich\-m\"ul\-ler space $T$ is a universal parameter space for 
the complex structures on all Riemann surfaces and can be defined as the space of all normalized quasisymmetric homeomorphisms on $\mathbb{S}$, namely, 
$T= \mbox{\rm M\"ob}(\mathbb{S}) \backslash \rm QS$.
A topology of $T$ can be defined by quasisymmetry constants of quasi\-symmetric homeo\-morphisms.
There are several ways to introduce quasisymmetric homeomorphisms on $\mathbb S$; in this paper, we 
lift $h:\mathbb{S} \to \mathbb{S}$
to $\tilde h:\mathbb{R} \to \mathbb{R}$ against the universal covering projection $\mathbb{R} \to \mathbb{S}$ with $x \mapsto e^{ix}$
and apply the definition of the quasisymmetry on $\mathbb R$ given in the next subsection.

The universal Teich\-m\"ul\-ler space $T$ can be also defined by using quasiconformal homeo\-morphisms $f$ of the unit disk 
$\mathbb{D}=\{z \in \mathbb{C} \mid |z|<1\}$ with complex dilatations $\mu_f=\bar \partial f/\partial f$
in the open unit ball $M(\mathbb{D})$ of the Banach space $L^{\infty}(\mathbb{D})$ 
of essentially bounded measurable functions on the unit disk $\mathbb{D}$. 
More precisely, for $\mu \in M(\mathbb{D})$, the solution of the
Beltrami equation (the measurable Riemann mapping theorem (see \cite{Ah66})) gives 
the unique quasiconformal homeomorphism $f^{\mu}$ of $\mathbb{D}$ onto itself that has complex dilatation $\mu$ 
and satisfies a certain normalization condition. 
This condition can be given by fixing three distinct points on $\mathbb{S}$,
for example, $1, i, -1$. We note that $f^\mu$ extends continuously to $\mathbb S$; the fixed point condition above
is applied to this extension.
This normalization cancels the freedom of post-composition of M\"obius transformations. 
By giving the normalization, $M(\mathbb{D})$ becomes a group with operation $*$,
where $\mu \ast \nu$ for $\mu, \nu \in M(\mathbb{D})$ is defined as the complex dilatation of $f^\mu \circ f^\nu$.
The inverse $\nu^{-1}$ denotes the complex dilatation of $(f^\nu)^{-1}$.

It is known that the continuous extension of $f^\mu$ to $\mathbb{S}$, denoted by the same $f^\mu$,
is a quasi\-symmetric homeo\-morphism of $\mathbb S$.
Conversely, any normalized (i.e., keeping the points $1, i, -1$ fixed) quasi\-symmetric homeo\-morphism of $\mathbb S$ extends continuously to a quasiconformal homeo\-morphism $f^\mu$ of $\mathbb D$ for some $\mu \in M({\mathbb D})$.
We say that $\mu$ and $\nu$ in $M(\mathbb{D})$ are equivalent ($\mu \sim \nu$), 
if $f^{\mu}=f^{\nu}$ on the unit circle $\mathbb{S}$. We denote the equivalence class of $\mu$ by $[\mu]$. Then,
the correspondence $[\mu] \mapsto f^{\mu}|_{\mathbb{S}}$ establishes a bijection 
from $M(\mathbb{D})/{\sim}$ onto $T$. Thus, the universal Teich\-m\"ul\-ler space $T$
is identified with $M(\mathbb{D})/{\sim}$. 
The topology of $T=\mbox{\rm M\"ob}(\mathbb{S}) \backslash \rm QS$ 
coincides with the quotient topology induced by the {\it Teich\-m\"ul\-ler projection} $\pi:M(\mathbb D) \to T$.

The universal Teich\-m\"ul\-ler space $T$ is also identified with a domain in the Banach space 
$$
B(\mathbb{D}^{*})=\{\varphi \mid \Vert \varphi \Vert_B={\rm sup}_{z \in \mathbb D^*} \rho_{\mathbb{D}^*}^{-2}(z)|\varphi(z)|<\infty\}
$$ 
of bounded holomorphic quadratic differentials $\varphi=\varphi(z)dz^2$ on 
$\mathbb{D}^{*}=\widehat{\mathbb C}-\overline{\mathbb{D}}$ under the {\it Bers embedding} $\beta: T \to B(\mathbb{D}^{*})$.
Here, $\rho_{\mathbb{D}^*}(z)=(|z|^2-1)^{-1}$ denotes the hyperbolic density on $\mathbb{D}^*$.
This map is given by the factorization of a map 
$\Phi:M(\mathbb D) \to B(\mathbb{D}^{*})$ by the Teich\-m\"ul\-ler projection $\pi$, i.e.,
$\beta \circ \pi=\Phi$. Here, for every $\mu \in  M(\mathbb D)$, $\Phi(\mu)$ is defined by the Schwarzian derivative 
${\mathcal S}(f_\mu|_{\mathbb D^*})$, where $f_{\mu}$ is a quasiconformal homeomorphism of the complex plane $\widehat{\mathbb{C}}$ that has complex dilatation $\mu$ in $\mathbb{D}$ and is conformal in $\mathbb{D}^*$. The map $\Phi$ is called the {\it Bers Schwarzian derivative map}.

The Bers embedding $\beta: T \to B(\mathbb{D}^{*})$ is a homeomorphism onto the image $\beta(T)=\Phi(M(\mathbb D))$, 
and it defines a complex structure of $T$
as a domain in the Banach space $B(\mathbb{D}^{*})$.
It is proved that $\Phi$ (and so is $\pi$) is a holomorphic split submersion from $M(\mathbb{D})$ onto its image. 

The barycentric extension due to Douady and Earle \cite{DE} gives a quasiconformal extension $E(h) \in {\rm QC}(\mathbb{D})$ of
a quasisymmetric homeomorphism $h \in {\rm QS}$ in a conformally natural way.
This means that $E(g_1 \circ h \circ g_2)=E(g_1) \circ E(h) \circ E(g_2)$ is satisfied for any $h \in {\rm QS}$ and 
any $g_1, g_2 \in \mbox{\rm M\"ob}(\mathbb{S})$, where the extensions $E(g_1)$ and $E(g_2)$ are conformal (M\"obius) on $\mathbb{D}$.
The quasiconformal extension $E(h)$ is a diffeomorphism of $\mathbb{D}$ that is bi-Lipschitz with respect to the hyperbolic metric.
The barycentric extension
induces a continuous (in fact, real analytic) section $s: T \to M(\mathbb{D})$ of the Teich\-m\"ul\-ler projection $\pi:M(\mathbb{D}) \to T$ ($\pi \circ s={\rm id}_T$) by sending a point $[\mu] \in T$ to the complex dilatation $s([\mu]) \in M(\mathbb{D})$ of $E(f^{\mu}|_{\mathbb{S}})$. 
By the conformal naturality of the barycentric extension,
the Teich\-m\"ul\-ler space of any Fuchsian group is shown to be contractible. 

\subsection{Symmetric Teich\-m\"ul\-ler space on the real line}
An increasing homeomorphism $h$ of the real line  $\mathbb{R}$ onto itself is said to be {\it quasisymmetric} if there exists some $M \geq 1$ such that 
$$
\frac{1}{M} \leq \frac{h(x+t)-h(x)}{h(x)-h(x-t)}\leq M
$$
for all $x\in\mathbb{R}$ and $t>0$. The optimal value of such $M$ is called the quasisymmetry constant for $h$.

Beurling and Ahlfors \cite{BA} proved that $h$ is quasisymmetric if and only if there exists some quasiconformal homeomorphism of the upper half-plane $\mathbb{U}=\{x+iy \in \mathbb{C} \mid y>0\}$ onto itself that is continuously extendable to the boundary map $h$. 
Let $\rm QS(\mathbb{R})$ denote the group of all quasisymmetric homeomorphisms of the real line $\mathbb{R}$.

A quasisymmetric homeomorphism $h$ is said to be {\it symmetric} if 
$$ 
\lim_{t\to 0+}\frac{h(x+t)-h(x)}{h(x)-h(x-t)}=1 
$$
uniformly for all $x\in\mathbb{R}$.
Let $\rm QS_{*}(\mathbb{R})$ denote the subset of $\rm QS(\mathbb{R})$ (in fact, this is not a subgroup
as shown in \cite{WM}) consisting of all symmetric homeomorphisms of the real line $\mathbb{R}$. It is known that $h$ is symmetric if and only if $h$ can be extended to an asymptotically conformal homeomorphism $f$ 
of the upper half-plane $\mathbb{U}$ onto itself (see \cite{GS}). 
In fact, the Beurling--Ahlfors extension of $h$ is asymptotically conformal when $h$ is symmetric. 
By an {\it asymptotically conformal} homeomorphism $f$ of the upper half-plane $\mathbb{U}$, we mean that its complex dilatation 
$\nu_f=\bar \partial f/\partial f$ belongs to $M_*(\mathbb{U})=L_*(\mathbb{U}) \cap M(\mathbb{U})$, where
$M(\mathbb{U})$ is the open unit ball of $L^\infty(\mathbb{U})$ and
$$
L_*(\mathbb{U})=\{ \nu \in L^\infty (\mathbb{U}) \mid  \forall\; \epsilon > 0, \exists\; t > 0\; {\rm such\; that}\; 
\Vert \nu|_{\mathbb{U} \setminus H_t}\Vert_{\infty} < \epsilon \}.
$$
Here, $H_t = \{x+iy \in \mathbb{U} \mid y \geq t\}$ $(t>0)$ is a horoplane tangent at $\infty$. 

We define $T_*(\mathbb{R})= \rm  Aff(\mathbb{R}) \backslash QS_*(\mathbb{R})$ as the {\it symmetric Teich\-m\"ul\-ler space on the real line} $\mathbb R$, where ${\rm Aff(\mathbb{R})}$ denotes
the subgroup of all real affine mappings $z \mapsto az+b$, $a>0, b\in\mathbb{R}$. 
Recently, Hu, Wu, and Shen \cite{HWS} endows $T_*(\mathbb{R})$ 
with a complex Banach manifold structure modeled on the closed subspace $B_*(\mathbb{L})$ 
of the Banach space 
$$
B(\mathbb{L})=\{\psi \mid \Vert \psi \Vert_B={\rm sup}_{z\in \mathbb{L}}\rho_{\mathbb{L}}^{-2}(z)|\psi(z)|<\infty \}
$$ 
of bounded holomorphic quadratic differentials on the lower half-plane $\mathbb{L}$, 
which consists of those $\psi$ 
satisfying that  
for any $\varepsilon > 0$, there exists $t>0$ such that 
$$
\sup_{z\in \mathbb{L} \setminus H_{t}^{*}}|\psi(z)|\rho_{\mathbb{L}}^{-2}(z) <\varepsilon.
$$
Here, $\rho_{\mathbb{L}}(z)=(2\,{\rm Im}\, z)^{-1}$ is the hyperbolic density on the lower half-plane 
$\mathbb{L}=\{x+iy \in \mathbb{C} \mid y<0\}$ and  
$H_{t}^{*}=\{x+iy \in \mathbb{L}\mid -y \geqslant t\}$ 
is the reflection of the horoplane $H_{t}$ as above with respect to the real line $\mathbb{R}$.

A quasisymmetric homeomorphism $h \in {\rm QS}$ on $\mathbb{S}$ is also called symmetric if its lift $\tilde h:\mathbb{R} \to \mathbb{R}$ is symmetric on $\mathbb{R}$ in the above sense.
We denote the subgroup of $\rm QS$ consisting of all symmetric homeomorphisms of $\mathbb{S}$ by $\rm Sym$.
Then,  the little universal Teich\-m\"ul\-ler space
was defined by $T_0=\mbox{\rm M\"ob}(\mathbb{S}) \backslash \rm Sym$, and have been studied
in the theory of asymptotic Teich\-m\"ul\-ler space (see \cite{EGL, GS}).
The universal Teich\-m\"ul\-ler space can be also defined on the real line 
by $T(\mathbb{R})=\rm  Aff(\mathbb{R}) \backslash QS(\mathbb{R})$ and this is isomorphic to
$T=\mbox{\rm M\"ob}(\mathbb{S}) \backslash \rm QS$ under the conjugation by the Cayley transformation. 
We note however that $T_*({\mathbb R})$ is not isomorphic to $T_0$ under this isomorphism $T(\mathbb{R}) \cong T$
(see \cite{HWS}).


\section{Generalized symmetric Teich\-m\"ul\-ler space}

In this section, we will introduce generalized symmetric homeomorphisms and 
the generalized symmetric Teich\-m\"ul\-ler space by transferring 
$\rm QS_*(\mathbb{R})$ to the unit circle $\mathbb{S}$ and extending it to the general case by
specifying a set of points $\xi\in\mathbb{S}$. 
We also transfer $M_*(\mathbb{U})$ to the unit disk $\mathbb{D}$ and $B_*(\mathbb{L})$ to the exterior of the unit disk $\mathbb{D}^*$. 

For every $\xi \in \mathbb{S}$,
let $\phi_{\xi}(z) = \xi (z-i)/(z+i)$ be the Cayley transformation of $\mathbb{U}$ onto $\mathbb{D}$ 
that sends $\infty$ to $\xi$ and $i$ to $0$. 
Then, the push-forward operator $(\phi_{\xi})_{*}:L^\infty(\mathbb{U}) \to L^\infty(\mathbb{D})$ defined by
$$
(\phi_{\xi})_{*}\nu =
\nu\circ \phi_{\xi}^{-1}\overline{(\phi_{\xi}^{-1})'}/(\phi_{\xi}^{-1})' 
$$
for every $\nu \in L^\infty(\mathbb{U})$ is a linear isometry. Let $L_*^{\xi}(\mathbb{D})=(\phi_{\xi})_{*}(L_*(\mathbb{U}))$ and
$M_*^{\xi}(\mathbb{D})=(\phi_{\xi})_{*}(M_*(\mathbb{U}))$. Clearly, $M_*^{\xi}(\mathbb{D})=L_*^{\xi}(\mathbb{D}) \cap M(\mathbb{D})$.

For our purpose of generalization, we represent $L_*^{\xi}(\mathbb{D})$ as follows.
We consider a horoball $D_{t}^{\xi} = \phi_{\xi}(H_t)$ for $t>0$, 
which is tangent at $\xi$ in $\mathbb{D}$ with the boundary 
$$
\partial D_{t}^{\xi} = \left\{\xi\,\frac{x+i(t-1)}{x+i(t+1)} \in {\mathbb D}\, \middle|\, x\in\mathbb{R}\right\}.
$$
Then, it is clear that 
$$
L_*^{\xi}(\mathbb{D})=\{\mu\in L^{\infty}(\mathbb{D})\mid  \forall\; \epsilon > 0, \exists\; t > 0\; {\rm such\; that}\; 
\Vert \mu|_{\mathbb{D} \setminus D_t^{\xi}}\Vert_{\infty} < \epsilon\}.
$$
 
Now we extend the definition above for only one tangent point to the case of plural tangent points. 
Let $X=\{\xi_1, \xi_2, \dots, \xi_n\}$ be a finite subset of $\mathbb{S}$. Let 
$$
L_*^{X}(\mathbb{D})=\{\mu\in L^{\infty}(\mathbb{D})\mid  \forall\; \epsilon > 0, \exists\; t > 0\; {\rm such\; that}\; 
\Vert \mu|_{\mathbb{D} \setminus \bigcup_{i=1}^{n}D_t^{\xi_i}}\Vert_{\infty} < \epsilon\},
$$
and $M_{*}^{X}(\mathbb{D})=L_{*}^{X}(\mathbb{D})\cap M(\mathbb{D})$. 
We see that $L_*^X(\mathbb{D})$ is closed in $L^{\infty}(\mathbb{D})$. Indeed, 
assuming that a sequence $\{\mu_k\}_{k \in \mathbb{N}}$ in $L_*^X(\mathbb{D})$ and $\mu \in L^{\infty}(\mathbb{D})$ are 
given so that $\Vert\mu_k-\mu\Vert_{\infty}\to 0$ as $k\to\infty$, we show that $\mu\in L_*^X(\mathbb{D})$.
For each $\varepsilon>0$, we can choose some $k_0 \in\mathbb{N}$ 
such that $\Vert\mu_{k_0}-\mu\Vert_{\infty}< \varepsilon$. Since $\mu_{k_0}\in L_*^X(\mathbb{D})$, 
there exists some $t>0$ such that $\Vert\mu_{k_0}|_{\mathbb{D} \setminus \bigcup_{i=1}^{n}D_{t}^{\xi_i}}\Vert_{\infty}<\varepsilon$. Thus, 
$$
\Vert\mu|_{\mathbb{D} \setminus \bigcup_{i=1}^{n}D_{t}^{\xi_i}}\Vert_{\infty} \leqslant \Vert\mu_{k_0}|_{\mathbb{D} \setminus \bigcup_{i=1}^{n}D_{t}^{\xi_i}}\Vert_{\infty}+\Vert\mu_{k_0}-\mu\Vert_{\infty}<2\varepsilon,
$$
which implies that $\mu\in L_*^X(\mathbb{D})$.

Here, we note the following fact on an algebraic structure of the space of the Beltrami differentials.

\begin{proposition}\label{L}
For any $X=\{\xi_1,\dots \xi_n\} \subset \mathbb {S}$,
$L_{*}^{X}(\mathbb{D})=L_{*}^{\xi_1}(\mathbb{D})+\dots +L_{*}^{\xi_n}(\mathbb{D})$.
\end{proposition}
\begin{proof}
The inclusion $\supset$ is easy to see. 
For the inverse inclusion $\subset$, we take any element $\mu$ in $L_{*}^{X}(\mathbb{D})$. 
The unit circle $\mathbb{S}$ is divided into $n$ sub-arcs by the points $\xi_1, \dots, \xi_n$. 
We take the midpoint of each sub-arc and connect it to the origin $0$ by a segment. 
The union of these segments divide $\mathbb{D}$ into $n$ sectors $E_1, \dots, E_n$, and
each $E_i$ $(i=1,\dots, n)$ contains only one $\xi_i$ on its boundary. 
Then, the decomposition of $\mu$ is given simply by restricting $\mu$ to each sector;
$\mu=\mu 1_{E_1}+ \dots +\mu 1_{E_n}$, where $\mu 1_{E_i} \in L_{*}^{\xi_i}(\mathbb{D})$ for each $i=1,\dots,n$.
\end{proof}

This implies that
$$
M_{*}^{X}(\mathbb{D})=(L_{*}^{\xi_1}(\mathbb{D})+\dots +L_{*}^{\xi_n}(\mathbb{D}))\cap M(\mathbb{D}).
$$

For $\mu \in M_{*}^{X}(\mathbb{D})$, a quasisymmetric homeomorphism
obtained by the boundary extension of a quasiconformal homeomorphism of $\mathbb{D}$ onto itself with
dilatation $\mu$ is called a {\it generalized symmetric homeomorphism} for $X$.
The subset of $\rm QS$ consisting of all such elements is denoted by ${\rm QS}_*^X$.
We remark that this is not a subgroup of $\rm QS$.

\begin{definition}
{\rm
Let $X \subset \mathbb{S}$ be a finite subset. The {\it generalized symmetric Teich\-m\"ul\-ler space} $T_*^X$
for $X$ is defined as 
$$
T_*^X=\mbox{\rm M\"ob}(\mathbb{S}) \backslash {\rm QS}_*^X=\pi(M_*^X(\mathbb{D})).
$$
}
\end{definition}

\begin{remark}
{\rm
The Teich\-m\"ul\-ler spaces $T$ and $T_0$ have a group structure by
the composition of the normalized elements of $\rm QS$. However, $T_*^X$ is not a subgroup of $T$
even if $X$ consists of only one point. See \cite{WM}.
}
\end{remark}

Although we will not pursue the characterization of generalized symmetric homeomorphisms $h$ for $X$
as the mapping on $\mathbb{S}$, we can expect the following claim.
For each interval $I \subset \mathbb{S}$ between consecutive points of $X$,
by stretching the map $h:I \to h(I)$ linearly and giving rotations to both sides, we obtain its conjugate 
$\check h_I:\mathbb{S}-\{1\} \to \mathbb{S}-\{1\}$. Then,  
$\phi_1^{-1} \circ \check h_I \circ \phi_1: \mathbb{R} \to \mathbb{R}$ is symmetric for every $I$ 
if and only if $h \in {\rm QS}_*^X$.


\section{Bers Schwarzian derivative map}

In this section, we focus on the Bers Schwarzian derivative map $\Phi:M(\mathbb{D}) \to B(\mathbb{D}^*)$
restricted to the subspace $M_*^X(\mathbb{D})$. We first introduce the corresponding subspace of $B(\mathbb{D}^*)$.

For this purpose, we use the same Cayley transformation $\phi_{\xi}(z) = \xi (z-i)/(z+i)$ as before for every $\xi \in \mathbb{S}$.
This also maps $\mathbb{L}$ onto $\mathbb{D}^*$ 
sending $\infty$ to $\xi$ and $-i$ to $\infty$. 
The push-forward operator $(\phi_{\xi})_{*}:B(\mathbb{L}) \to B(\mathbb{D}^*)$ defined by
$$
(\phi_{\xi})_{*}\psi = \psi\circ \phi_{\xi}^{-1}(\phi_{\xi}^{-1})'^{2} 
$$
for every $\psi \in B(\mathbb{L})$ is a linear isometry. 
Let $B_{*}^{\xi}(\mathbb{D}^*)=(\phi_{\xi})_{*}(B_*(\mathbb{L}))$.

We consider a horoball $(D_{t}^{\xi})^* = \phi_{\xi}(H_{t}^{*})$ tangent at $\xi$ in $\mathbb{D}^{*}$ 
such that 
$$
\partial (D_{t}^{\xi})^*=\left \{\xi\,\frac{x+i(t+1)}{x+i(t-1)} \in \mathbb{D}^*\, \middle|\, x \in\mathbb{R}\right\}.
$$ 
This is the reflection of $D_{t}^{\xi}$ with respect to $\mathbb{S}$. 
Then, we see that
$$
B_*^{\xi}(\mathbb{D}^*)=\{\varphi\in B(\mathbb{D}^{*})\mid  \forall\; \varepsilon > 0, \exists\; t > 0\; {\rm such\; that}\; 
\Vert \varphi|_{\mathbb{D}^* \setminus (D_{t}^{\xi})^*} \Vert_B < \varepsilon\}.
$$
For a finite subset $X=\{\xi_1, \xi_2, \dots, \xi_n\} \subset \mathbb{S}$, 
we also generalize this to 
$$
B_*^{X}(\mathbb{D}^*)=\{\varphi\in B(\mathbb{D}^{*})\mid  \forall\; \varepsilon > 0, \exists\; t > 0\; {\rm such\; that}\; 
\Vert \varphi|_{\mathbb{D}^* \setminus \bigcup_{i=1}^{n}(D_{t}^{\xi_i})^*}\Vert_B < \varepsilon\},
$$
which is the desired Banach subspace of $B(\mathbb{D}^*)$.

By the following theorem, we see that $B_*^X(\mathbb{D}^*)$ is the appropriate space 
corresponding to $M_*^X(\mathbb{D})$
under the Bers Schwarzian derivative map $\Phi$.

\begin{theorem}\label{Bers}
For every finite subset $X \subset \mathbb{S}$,
the Bers Schwarzian derivative map $\Phi$ maps $M_*^X(\mathbb{D})$ into $B_*^X(\mathbb{D}^*)$. 
\end{theorem}
\begin{proof}
By the integral representation of the Schwarzian derivative, which was established by Astala and Zinsmeister \cite{AZ}
(see also Cui \cite{Cu}), we have that for $\zeta^* \in \mathbb{D}^*$, 
$$
\rho_{\mathbb{D}^*}^{-4}(\zeta^*)|\Phi(\mu)(\zeta^*)|^2 \leqslant C\int_{\mathbb{D}}\frac{(|\zeta^*|^2-1)^2}
{|z-\zeta^*|^4}|\mu(z)|^2dxdy, 
$$
where $C>0$ is a constant depending only on 
$\Vert\mu\Vert_{\infty}.$ 

Let $\gamma_{\zeta}(z)=(\overline{\zeta^*}z-1)/(z-\zeta^*) \in \mbox{\rm M\"ob}(\mathbb{D})$ 
be a M\"obius transformation of $\mathbb{D}$ onto itself that sends $\zeta$ to $0$. 
Here, $\zeta \in \mathbb{D}$ and $\zeta^* \in \mathbb{D}^*$ are the reflection to each other with respect to $\mathbb{S}$. 
We see that $|\gamma_{\zeta}'(z)|^2=(|\zeta^*|^2-1)^2/|z-\zeta^*|^4$. It follows that 
\begin{equation*}
 \begin{split}
 &\quad \int_{\mathbb{D}}\frac{(|\zeta^*|^2-1)^2}
{|z-\zeta^*|^4}|\mu(z)|^2dxdy  =   \int_{\mathbb{D}}|\gamma_{\zeta}'(z)|^2|\mu(z)|^2dxdy\\
&=\int_{\mathbb{D} \setminus \bigcup_{i=1}^{n}D_{t}^{\xi_i}}|\gamma_{\zeta}'(z)|^2|\mu(z)|^2dxdy + \sum_{i=1}^{n}\int_{D_{t}^{\xi_i}}|\gamma_{\zeta}'(z)|^2|\mu(z)|^2dxdy.\\
 \end{split}   
\end{equation*}
Here, for a given $\varepsilon > 0$, we choose $t>0$
so that $\Vert \mu|_{\mathbb{D} \setminus \bigcup_{i=1}^{n}D_{t}^{\xi_i}} \Vert_{\infty} < \varepsilon$
under the condition $\mu \in M_*^X(\mathbb{D})$.
Then, the last formula is estimated from above by 
\begin{equation*}
 \begin{split}
 &\quad\ \varepsilon^2 \int_{\mathbb{D} \setminus \bigcup_{i=1}^{n}D_{t}^{\xi_i}}|\gamma_{\zeta}'(z)|^2dxdy + \sum_{i=1}^{n}\int_{D_{t}^{\xi_i}}|\gamma_{\zeta}'(z)|^2dxdy \\
 &\leqslant\pi\varepsilon^2 + \sum_{i=1}^{n}{\rm Area}(\gamma_{\zeta}(D_{t}^{\xi_i})),\\
 \end{split}   
\end{equation*}
where $\rm Area$ stands for the Euclidean area. 

We consider ${\rm Area}(\gamma_{\zeta}(D_{t}^{\xi_i}))$.
The notation $\asymp$ is used below when the both sides are 
comparable, i.e., one side is bounded from above and below by multiples of the other side
with some positive absolute constants.
By ${\rm Area}(\gamma_{\zeta}(D_t^{\xi_i}))\asymp {\rm diam^2}(\gamma_{\zeta}(D_t^{\xi_i}))$ and 
$d_H(0,z)=\log\frac{1+|z|}{1-|z|}$ $(z \in \mathbb{D})$ for the Euclidean diameter $\rm diam$ and the hyperbolic distance $d_H$,
we see that 
\begin{equation*}
 {\rm diam^2}(\gamma_{\zeta^{*}}(D_t^{\xi_i}))
 \asymp e^{-2d_H(0, \gamma_{\zeta}(D_{t}^{\xi_i}))} = e^{-2d_H(\zeta, D_{t}^{\xi_i})}.
\end{equation*}
Therefore, the condition ${\rm Area}(\gamma_{\zeta}(D_{t}^{\xi_i})) \leqslant \varepsilon^2$ is equivalent to that $d_H(\zeta, D_{t}^{\xi_i}) \geqslant -\log \varepsilon$ up to some multiple constant. 
We note that $d_H(\zeta, D_t^{\xi_i})=d_H(\zeta^*, (D_t^{\xi_i})^*)$ by reflection, and that the hyperbolic 
$a$-neighborhood $N_a(D_{t}^{\xi_i})$ of $D_{t}^{\xi_i}$ is 
$$
N_a(D_{t}^{\xi_i}) = h_{\xi_i}(N_a(H_{t}))=h_{\xi_i}(H_{e^{-a}t}) = D_{e^{-a}t}^{\xi_i}.
$$
Thus, 
the condition $d_H(\zeta, D_{t}^{\xi_i}) \geqslant - \log \varepsilon$ is equivalent to 
that $\zeta^* \notin (D_{\varepsilon t}^{\xi_i})^*$.
This implies that if $\zeta^* \in \mathbb{D}^* \setminus \bigcup_{i=1}^{n}(D_{\varepsilon t}^{\xi_i})^*$, then 
${\rm Area}(\gamma_{\zeta}(D_{t}^{\xi_i}))\leqslant A\varepsilon^2$ for some absolute constant $A$.

We plug this area estimate in the above inequality.
The conclusion is that 
if $\zeta^* \in \mathbb{D}^* \setminus \bigcup_{i=1}^{n}(D_{\varepsilon t}^{\xi_i})^*$, then 
$$
\rho_{\mathbb{D}^*}^{-2}(\zeta^*)|\Phi(\mu)(\zeta^*)| \leqslant \sqrt{C(\pi+An)}\,\varepsilon.
$$
Since $\varepsilon>0$ is arbitrarily chosen, this implies that if $\mu \in M_*^X(\mathbb{D})$, then
$\Phi(\mu) \in B_*^X(\mathbb{D}^*)$.
 \end{proof}
 
We note that $\Phi:M_*^X(\mathbb{D}) \to B_*^X(\mathbb{D}^*)$ is holomorphic because
$\Phi:M(\mathbb{D}) \to B(\mathbb{D}^*)$ is holomorphic and the closed subspaces 
$M_*^X(\mathbb{D})$ and $B_*^X(\mathbb{D}^*)$ are endowed with
the relative topologies from $M(\mathbb{D})$ and $B(\mathbb{D}^*)$, respectively.

 
\section{Barycentric extension}

In this section, we will prove that the barycentric extension due to Douady and Earle \cite{DE}
gives an appropriate right inverse of $\pi:M_{*}^{X}(\mathbb{D}) \to T_*^X$
from the generalized symmetric Teich\-m\"ul\-ler space $T_*^X$ to
the space $M_{*}^{X}(\mathbb{D})$ of complex dilatations. In other words, for the section $s:T \to M(\mathbb{D})$ of the universal Teich\-m\"ul\-ler space induced by the barycentric extension,
we will show that the image $s(T_*^X)$ is in $M_{*}^{X}(\mathbb{D})$.

This claim follows from the following more general result concerning the section $s$.
This was originally proved by Earle, Markovic, and Saric \cite[Theorem 4]{EMS} for the little universal  Teich\-m\"ul\-ler space 
$T_0=\mbox{\rm M\"ob}(\mathbb{S}) \backslash {\rm Sym}$ and for
the subspaces
$M_0(\mathbb{D}) \subset M(\mathbb{D})$ and $B_0(\mathbb{D}^*) \subset B(\mathbb{D}^*)$
consisting of the vanishing elements at the boundary.
The proof below is a modification of theirs.

\begin{theorem}\label{EMS}
Let $\mu$ and $\nu$ be in $M(\mathbb{D})$. 
Let $X$ be a finite subset of $\mathbb{S}$.
Then, the following are equivalent:
\begin{enumerate}
\item
$s([\mu]) - s([\nu])\in L_*^X(\mathbb{D})$;
\item
$\Phi(\mu)-\Phi(\nu)\in B_*^X(\mathbb{D}^*)$.
\end{enumerate}
\end{theorem} 
\begin{proof}
$(2) \Rightarrow (1)$: We take an arbitrary sequence $\{z_k\}_{k \in \mathbb {N}} \subset \mathbb{D}$ such that
$z_k \in \mathbb{D} \setminus \bigcup_{i=1}^n D_{1/k}^{\xi_i}$ for every $k \in \mathbb {N}$.
For each $k$, we choose a M\"obius transformation $g_k \in \mbox{\rm M\"ob}(\mathbb{D})$ with $g_k(0)=z_k$,
and define $\mu_k=g_k^* s([\mu])$ and $\nu_k=g_k^* s([\nu])$. Then, $\Phi(\mu_k)=g_k^* \Phi(\mu)$ and 
$\Phi(\nu_k)=g_k^* \Phi(\nu)$ for $g_k \in \mbox{\rm M\"ob}(\mathbb{D}^*)$. 
We also see that for every $\tilde k \in \mathbb N$ and for every $z^* \in \mathbb {D}^*$, there is some $k_0$ such that 
$g_k(z^*) \in {\mathbb D}^* \setminus \bigcup_{i=1}^n (D_{1/\tilde k}^{\xi_i})^*$ for all $k \geq k_0$. 
Since we assume that $\Phi(\mu)-\Phi(\nu) \in B_*^X(\mathbb{D}^*)$, we have that
$$
\rho_{\mathbb{D}^*}^{-2}(z^*)|\Phi(\mu_k)(z^*)-\Phi(\nu_k)(z^*)|
=\rho_{\mathbb{D}^*}^{-2}(g_k(z^*))|(\Phi(\mu)-\Phi(\nu))(g_k(z^*))| 
$$
tends to $0$ as $k \to \infty$ for each $z^* \in \mathbb{D}^*$.
In particular, $\Phi(\mu_k)-\Phi(\nu_k) \to 0$ as $k \to \infty$.

Since $\Vert \mu_k \Vert_\infty=\Vert s([\mu]) \Vert_\infty$ and $\Vert \nu_k \Vert_\infty=\Vert s([\nu]) \Vert_\infty$,
by passing to a subsequence, we may assume that
$f^{\mu_{k}}$ converges uniformly to some 
quasiconformal homeomorphism $f^{\mu_0} \in {\rm QC}({\mathbb D})$ with a complex dilatation $\mu_0 \in M({\mathbb D})$ and $f^{\nu_{k}}$ 
converges uniformly to some 
$f^{\nu_0} \in {\rm QC}({\mathbb D})$ with $\nu_0 \in M({\mathbb D})$. In this situation,
\cite[Lemma 6.1]{EMS} asserts that 
$\Phi(\mu_{k})$ converges locally uniformly to
$\Phi(\mu_0)$ and $\Phi(\nu_{k})$ converges locally uniformly to
$\Phi(\nu_0)$ on $\mathbb{D}^*$.  
Since $\Phi(\mu_k)-\Phi(\nu_k) \to 0$ as $k \to \infty$, this implies that $\Phi(\mu_0)=\Phi(\nu_0)$.

By \cite[Lemma 6.1]{EMS} again, we see that $s([\mu_{k}])$ converges locally uniformly to
$s([\mu_0])$ and $s([\nu_{k}])$ converges locally uniformly to
$s([\nu_0])$ on $\mathbb{D}$. Here, $\Phi(\mu_0)=\Phi(\nu_0)$ implies that $s([\mu_0])=s([\nu_0])$.
Therefore, $s([\mu_{k}])-s([\nu_{k}])$ converges to $0$, and in particular,
$s([\mu_{k}])(0)-s([\nu_{k}])(0) \to 0$ as $k \to \infty$.

The conformal naturality of the barycentric extension implies that
$$
s([\mu_{k}])=s([g_k^*\mu])=g_k^*(s([\mu])); \quad s([\nu_{k}])=s([g_k^*\nu])=g_k^*(s([\nu])).
$$
It follows that
$$
|s([\mu])(z_k)-s([\nu])(z_k)|=|s([\mu_k])(0)-s([\nu_k])(0)| \to 0 \quad (k \to \infty).
$$
Since $z_k \in \mathbb{D} \setminus \bigcup_{i=1}^n D_{1/k}^{\xi_i}$ for every $k \in \mathbb {N}$,
we see that $s([\mu])-s([\nu]) \in L_*^X(\mathbb{D})$.

$(1) \Rightarrow (2)$:  
We take an arbitrary sequence $\{z_k\}_{k \in \mathbb {N}} \subset \mathbb{D}$ such that
$z_k \in \mathbb{D} \setminus \bigcup_{i=1}^n D_{1/k}^{\xi_i}$ for every $k \in \mathbb {N}$.
For each $k$, we choose a M\"obius transformation $g_k \in \mbox{\rm M\"ob}(\mathbb{D})$ with $g_k(0)=z_k$,
and define $\mu_k=g_k^* s([\mu])$ and $\nu_k=g_k^* s([\nu])$. Then, 
$$
\Vert (\mu_k-\nu_k)|_{\Delta(0,r)} \Vert_\infty=\Vert (s([\mu])- s([\nu]))|_{\Delta(z_k,r)} \Vert_\infty
$$
tends to $0$ as $k \to \infty$ for any $r>0$. Here, $\Delta(a,r) \subset \mathbb{D}$ denotes a hyperbolic disk
with center $a$ and radius $r$.

Since $\Vert \mu_k \Vert_\infty=\Vert s([\mu]) \Vert_\infty$ and $\Vert \nu_k \Vert_\infty=\Vert s([\nu]) \Vert_\infty$,
by passing to a subsequence, we may assume that  
$f^{\mu_{k}}$ converges uniformly to some 
quasiconformal homeomorphism $f^{\mu_0} \in {\rm QC}({\mathbb D})$ with a complex dilatation $\mu_0 \in M({\mathbb D})$ and $f^{\nu_{k}}$ 
converges uniformly to some 
$f^{\nu_0} \in {\rm QC}({\mathbb D})$ with $\nu_0 \in M({\mathbb D})$. Let $\lambda_k=\mu_k \ast \nu_k^{-1}$, that is,
$\lambda_k$ is the complex dilatation of $f^{\mu_k} \circ (f^{\nu_k})^{-1}$. This satisfies
$$
|\lambda_k \circ f^{\nu_k}|=\frac{|\mu_k -\nu_k|}{|1-\overline{\nu_k} \mu_k|}.
$$

For an arbitrary compact subset $E \subset \mathbb{D}$, we take $r>0$ such that $(f^{\nu_0})^{-1}(E) \subset \Delta(0,r)$.
Since $(f^{\nu_{k}})^{-1}$ converges to $(f^{\nu_0})^{-1}$ uniformly on $\mathbb{D}$ as $k \to \infty$, we can assume that
$(f^{\nu_{k}})^{-1}(E) \subset \Delta(0,r)$ for all sufficiently large $k$. Hence,
$$
\Vert \lambda_{k}|_E \Vert_\infty \leq 
\frac{\Vert (\mu_{k}-\nu_{k})|_{\Delta(0,r)}\Vert_\infty}{1-\Vert \mu \Vert_\infty \Vert \nu \Vert_\infty}
\to 0 \quad (k \to \infty).
$$
Since $E$ is arbitrary, we see from this estimate that the limit $f^{\mu_0} \circ (f^{\nu_0})^{-1}$ of
$f^{\mu_{k}} \circ (f^{\nu_{k}})^{-1}$ is conformal on $\mathbb{D}$. In fact, $f^{\mu_0} \circ (f^{\nu_0})^{-1}$ is the identity
by the normalization. Therefore, $f^{\mu_0}= f^{\nu_0}$, and both $f^{\mu_{k}}$ and $f^{\nu_{k}}$
converge uniformly to the same limit $f^{\mu_0}$ as $k \to \infty$.

For every $\mu \in M(\mathbb D)$, we define $\widetilde \Phi(\mu)(z)=z^4\Phi(\mu)(z)$ ($z \in \mathbb D^*$).
As $\rho_{\mathbb D^*}^{-2}(z)|\Phi(\mu)(z)|$ is bounded, we see that $\widetilde \Phi(\mu)$ is a holomorphic function on $\mathbb D^*$.
Similarly to \cite[Lemma 6.1]{EMS}, it can be proved that 
$\widetilde \Phi(\mu_{k})$ and $\widetilde \Phi(\nu_{k})$ converge to the same limit $\widetilde \Phi(\mu_0)$
locally uniformly on $\mathbb{D}^*$ as $k \to \infty$. Therefore, $\widetilde \Phi(\mu_{k})-\widetilde \Phi(\nu_{k})$ 
converges to $0$, and in particular,
$\widetilde \Phi(\mu_{k})(\infty)-\widetilde \Phi(\nu_{k})(\infty) \to 0$ as $k \to \infty$.

The equivariance of the Bers projection implies that 
$$
\Phi(\mu_k)=\Phi(g_k^*\mu)=g_k^* \Phi(\mu);\quad \Phi(\nu_k)=\Phi(g_k^*\nu)=g_k^* \Phi(\nu).
$$
By $\lim_{z \to \infty}g_k(z)=z_k^*$ and $\lim_{z \to \infty}|z^2g'_k(z)|=\rho_{\mathbb{D}^*}^{-1}(z_k^*)$,
it follows that
\begin{align*}
\rho_{\mathbb{D}^*}^{-2}(z_k^*)|\Phi(\mu)(z_k^*)-\Phi(\nu)(z_k^*)|
&=\lim_{z \to \infty}|z^2g'_k(z)|^2|\Phi(\mu)(g_k(z))-\Phi(\nu)(g_k(z))|\\
&=|\widetilde\Phi(\mu_k)(\infty)-\widetilde\Phi(\nu_k)(\infty)|.
\end{align*}
This tends to $0$ as $k \to \infty$.
Since $z_k^* \in \mathbb{D}^* \setminus \bigcup_{i=1}^n (D_{1/k}^{\xi_i})^*$ are arbitrarily chosen,
we see that $\Phi(\mu)-\Phi(\nu) \in B_*^X(\mathbb{D})$.
\end{proof}

Here are direct consequences from this theorem.

\begin{corollary}\label{DE}
For every $h\in \rm QS_{*}^{X}$, the complex dilatation of the barycentric extension $E(h)$ is in $M_{*}^{X}(\mathbb{D})$.
Hence, we have a global continuous section $s:T_*^X \to M_{*}^{X}(\mathbb{D})$ to the 
Teich\-m\"ul\-ler projection $\pi:M_{*}^{X}(\mathbb{D}) \to T_*^X$.
\end{corollary}

\begin{proof}
By setting $\nu=0$ in Theorem \ref{EMS}, we obtain that 
$s([\mu]) \in M_{*}^{X}(\mathbb{D})$ is equivalent to that $\Phi(\mu) \in B_*^X(\mathbb{D}^*)$.
Let $\mu \in M_{*}^{X}(\mathbb{D})$ be the complex dilatation of some quasiconformal extension of $h$.
Then, the complex dilatation of the barycentric extension of $h$ is $s([\mu])$.
Since $\Phi(\mu) \in B_*^X(\mathbb{D}^*)$ by Theorem \ref{Bers}, we see that $s([\mu]) \in M_{*}^{X}(\mathbb{D})$.
\end{proof}

\begin{corollary}\label{contractible} 
The Teich\-m\"ul\-ler space $T_*^X$ is contractible.
\end{corollary}

\begin{proof}
Since $M_{*}^{X}(\mathbb{D})$ is contractible, the assertion follows from Corollary \ref{DE}.
\end{proof}

\begin{corollary}\label{image}
$\beta(T_*^X) = \beta(T)\cap B_*^X(\mathbb{D}^*)$.
\end{corollary}
 \begin{proof}
 Theorem \ref{Bers} implies that $\beta(T_*^X) \subset \beta(T)\cap B_*^X(\mathbb{D}^*)$. By taking $\nu = 0$ in Theorem \ref{EMS}, 
 we see that the converse inclusion is also true.
 \end{proof}
 

\section{Holomorphic split submersion}

In this section, we will endow $T_{*}^{X}$ with a complex Banach manifold structure. 
This is done by the investigations of the Bers Schwarzian derivative map $\Phi:M_*^X(\mathbb{D})\to B_*^X(\mathbb{D}^*)$
given in Theorem \ref{Bers} and the section $s$ to $\pi: M_*^X(\mathbb{D}) \to T_{*}^{X}$
induced by the barycentric extension in Corollary \ref{DE}. We note that the image of $\Phi$ is 
$\beta(T_*^X) = \beta(T)\cap B_*^X(\mathbb{D}^*)$ by Corollary \ref{image}, which is an open subset of $B_*^X(\mathbb{D}^*)$

We recall that the right translation $r_\nu$ for any $\nu \in M(\mathbb D)$ defined by
$r_\nu(\mu)=\mu \ast \nu^{-1}$ for every $\mu \in M(\mathbb D)$ is a biholomorphic automorphism of $M(\mathbb D)$.
Concerning the restriction of these automorphisms to $M_*^X(\mathbb D)$, we in particular obtain the following result
for the right translation $r_\nu$ given by a trivial Beltrami coefficient $\nu$, which satisfies 
$\pi \circ  r_\nu=\pi$ for the Teich\-m\"ul\-ler projection $\pi:\mu \mapsto [\mu]$.

\begin{lemma}\label{biholo}
Let $\nu \in M_*^X(\mathbb D)$ such that $[\nu]=[0]$. Then, $r_\nu$ is a biholomorphic automorphism of $M_*^X(\mathbb D)$.
\end{lemma}

\begin{proof}
We have only to prove that $r_\nu(\mu)$ belongs to $M_*^X(\mathbb D)$ for every $\mu \in M_*^X(\mathbb D)$.
The chain rule of complex dilatations implies that
$$
|r_\nu(\mu)\circ f^\nu(z)|=\frac{|\mu(z)-\nu(z)|}{|1-\overline{\nu(z)}\mu(z)|}
$$
for $z \in \mathbb D$. Then, it suffices to show that 
the image $f^\nu(D_t^{\xi})$ of a horoball $D_t^{\xi}$ for any $\xi \in X$ and $t>0$
is contained in a horoball $D_{t'}^{\xi}$ for some $t'>0$.

We may consider this problem on the upper half-plane $\mathbb U$ under the Cayley transformation
$\phi_\xi:\mathbb U \to \mathbb D$. Then, the horoball $D_t^{\xi} \subset \mathbb D$ corresponds to 
$H_t \subset \mathbb U$. Let $\tilde f_\nu=\phi_\xi^{-1} \circ f_\nu \circ \phi_\xi$, which extends to 
the boundary $\mathbb R$ as the identity. By some distortion theorem of quasiconformal maps, we can show that
there are constant $t',t''>0$ depending only on $t>0$ and $\Vert \nu \Vert_\infty$ with $t',t'' \to 0$
as $t \to 0$ such that
$H_{t''} \subset \tilde f_\nu(H_t) \subset H_{t'}$, the latter of which is our desired result.
For instance, we take any point $(x,t)$ on $\partial H_t$ and other three points $(x-t,0)$, $(x+t,0)$ and $(x,-t)$.
We may assume that $\tilde f_\nu$
is a quasiconformal self-homeomorphism of $\mathbb C$ by the reflection with respect to the real line.
Then, the distortion theorem of the cross ratio for four points
due to Teich\-m\"ul\-ler (see \cite[Chapter III.D]{Ah66}) implies that
there are such $t',t''>0$ satisfying $t'' \leq {\rm Im}\, f_\nu(x,t) \leq t'$ independently of $x \in \mathbb R$.
\end{proof}

We also see that any equivalent Beltrami coefficients $\mu_1, \mu_2 \in M_*^X(\mathbb D)$ 
are mapped to one another by a biholomorphic automorphism $r_\nu$ of $M_*^X(\mathbb D)$ for some trivial $\nu \in M_*^X(\mathbb D)$.

\begin{proposition}\label{trivial}
For any $\mu_1, \mu_2 \in M_*^X(\mathbb D)$ such that $[\mu_1]=[\mu_2]$,
the composition $\nu=\mu_1^{-1} \ast \mu_2$ belongs to $M_*^X(\mathbb D)$.
\end{proposition}

\begin{proof}
The condition $\nu=\mu_1^{-1} \ast \mu_2$ is equivalent to $r_\nu(\mu_2)=\mu_1$.
Then, we have that
$$
|\mu_1 \circ f^\nu(z)|=\frac{|\mu_2(z)-\nu(z)|}{|1-\overline{\nu(z)}\mu_2(z)|}
$$
for $z \in \mathbb D$.
Since $[\nu]=0$, the argument in the proof of Lemma \ref{biholo} concerning the image of a horoball by $f^\nu$
can be also applied to see that if $\mu_1, \mu_2 \in M_*^X(\mathbb D)$ then $\nu \in M_*^X(\mathbb D)$.
\end{proof}

With the aid of these claims, we can show that the Bers Schwarzian derivative map $\Phi$ 
is a holomorphic split submersion onto its image. We note that to endow the Teich\-m\"ul\-ler space
with the complex Banach manifold structure, 
it is enough only to show the existence of a local continuous section to $\Phi$
in our situation (see Corollary \ref{complexstructure} below).

\begin{theorem}\label{submersion}
The Bers Schwarzian derivative map $\Phi:M_*^X(\mathbb{D})\to B_*^X(\mathbb{D}^*)$ 
is a holomorphic split submersion onto its image $\Phi(M_*^X(\mathbb{D}))=\beta(T)\cap B_*^X(\mathbb{D}^*)$. 
\end{theorem}

\begin{proof}
Since $\Phi: M(\mathbb{D})\to B(\mathbb{D^*})$
is holomorphic and since $M_*^X(\mathbb{D})$ and $B_*^X(\mathbb{D}^*)$ are closed subspaces in the relative topology,
$\Phi:M_*^X(\mathbb{D})\to B_*^X(\mathbb{D}^*)$ is also holomorphic.
It remains to show that $\Phi$ is a split submersion onto its image $\Phi(M_*^X(\mathbb{D}))$.
This is equivalent to showing that for every $\mu \in M_*^X(\mathbb{D})$,
there is a holomorphic map $\sigma:U_\phi \to M_*^X(\mathbb{D})$ defined on some neighborhood $U_\phi \subset 
\Phi(M_*^X(\mathbb{D}))$ of $\phi=\Phi(\mu)$ such that $\sigma(\phi)=\mu$ and $\Phi \circ \sigma={\rm id}_{U_\phi}$.
The existence of some local holomorphic section can be given by a standard argument.
This has been carried out in \cite{HWS} in the case where $X$ is a single point set, and we repeat
such an argument adapted for our case below. 

In order to prove that $\Phi$ is a split submersion,
we supplement the proof in \cite{HWS} here by showing that for any $\mu \in M_*^X(\mathbb{D})$,
there is a local holomorphic section defined on a neighborhood of $\phi=\Phi(\mu)$ that sends $\phi$ to
$\mu$. We assume that there is a local holomorphic section $\sigma:U_\phi \to M_*^X(\mathbb{D})$.
We set $\nu=\mu^{-1} \ast \sigma(\phi)$, which belongs to $M_*^X(\mathbb{D})$ by Proposition \ref{trivial}. 
By Lemma \ref{biholo}, $r_\nu$ is a biholomorphic automorphism of $M_*^X(\mathbb{D})$ which satisfies $\pi \circ r_\nu=\pi$ and
$r_\nu(\sigma(\phi))=\mu$.
Then, we obtain the required local section $r_\nu \circ \sigma$ on $U_\phi$.

In the rest of the proof, we show the existence of a local holomorphic section. 
Let $\phi = \Phi(\mu)$ for a given $\mu \in M_*^X(\mathbb{D})$. 
Without loss of generality, we may assume that $\mu = s([\mu])$, that is, $f^{\mu}$ is the barycentric extension of $f^{\mu}|_{\mathbb{S}}$. Here, $s: T \to M(\mathbb{D})$ is the barycentric section 
which maps $T_*^X$ into $M_*^X(\mathbb{D})$ by Corollary \ref{DE}. For
the quasiconformal homeomorphism $f_\phi=f_{\mu}:\widehat {\mathbb{C}} \to \widehat {\mathbb{C}}$ that is conformal on $\mathbb{D}^*$,
we set
$D=f_\phi(\mathbb{D})$, $D^* = f_\phi(\mathbb{D}^*)$, and $\gamma = f_{\phi}\circ j \circ f_{\phi}^{-1}$
for the reflection $j:\zeta \mapsto \zeta^*$ with respect to $\mathbb{S}$. 
We may assume that $f_{\phi}$ is normalized so that $\lim_{z \to \infty}(f_\phi(z)-z)=0$.
Since the barycentric extension $f^\mu$ is a bi-Lipschitz diffeomorphism with respect to the hyperbolic metric,
we see that so is $f_{\phi}|_{\mathbb{D}}$, and hence, the quasiconformal reflection $\gamma: D \to D^*$ is
a bi-Lipschitz diffeomorphism with respect to the hyperbolic metrics on $D$ and $D^*$.

Ahlfors \cite{Ah63} (see also \cite{GL, Le})
showed that there exists a constant $C_1 \geq 1$ depending only 
on $\Vert\mu\Vert_{\infty}$ such that 
\begin{equation}\label{1}
\frac{1}{C_1}\leqslant |\gamma(z)-z|^2\rho_{D^*}^{-2}(\gamma(z))|\bar{\partial}\gamma(z)|\leqslant C_1
\end{equation}
for every $z\in D$, where $\rho_{D^*}(z)$ is the hyperbolic density on $D^*$. 
We set 
$$
B_{\varepsilon}(\phi)=\{\psi \in B_*^X(\mathbb{D}^*) \mid \Vert\psi - \phi\Vert_{B}< \varepsilon  \}
$$ 
for $\varepsilon > 0$. 
For each $\psi \in B_{\varepsilon}(\phi)$, 
there exists a unique locally univalent holomorphic function $f_{\psi}$ on $\mathbb{D}^*$ 
with the normalization as above such that 
$\mathcal{S}(f_{\psi})= \psi$. Let $g_{\psi} = f_{\psi}\circ f_{\phi}^{-1}|_{D^*}$. 
Then, we have that $\mathcal{S}(g_{\psi})\circ f_\phi (f'_\phi)^2 = \psi - \phi$ and 
$\sup_{z^*\in D^*}\rho_{D^*}^{-2}(z^*)|\mathcal{S}(g_{\psi})(z^*)|=\Vert\psi - \phi\Vert_{B}$.

When $\varepsilon>0$ is sufficiently small, it was proved in \cite{Ah63}
that $g_{\psi}$ is univalent (conformal) and can be extended to a quasiconformal homeomorphism of
$\widehat {\mathbb{C}}$ whose complex dilatation $\mu_{\psi}$ on $D$ has the form 
\begin{equation}\label{2}
 \mu_{\psi}(z)=\frac{\mathcal{S}(g_{\psi})(\gamma(z))(\gamma(z) - z)^2\bar{\partial}\gamma(z)}{2 + \mathcal{S}(g_{\psi})(\gamma(z))(\gamma(z) - z)^2\partial \gamma(z)}.
\end{equation}
We set $U_\phi=B_\varepsilon(\phi)$ for this $\varepsilon>0$.
Then by (\ref{1}), every $\psi \in U_\phi$ satisfies
\begin{equation}\label{3}
|\mu_{\psi}(z)|\leqslant C_2 |\mathcal{S}(g_{\psi})(\gamma(z))|\rho_{D^*}^{-2}(\gamma(z)) \qquad (z\in D)    
\end{equation}
for some constant $C_2>0$, which also depends only on $\Vert\mu\Vert_{\infty}$.

Consequently, $f_{\psi}=g_{\psi}\circ f_\phi$ is conformal on $\mathbb{D}^*$ 
and has a quasiconformal extension to $\widehat {\mathbb{C}}$ whose complex dilatation $\nu_{\psi}$ on $\mathbb{D}$ is given as
\begin{equation}\label{4}
    \nu_{\psi}=\frac{\mu+(\mu_{\psi}\circ f_\phi)\tau}{1+\bar{\mu}(\mu_{\psi}\circ f_\phi)\tau}, 
    \qquad \tau = \frac{\overline{\partial f_\phi}}{\partial f_\phi}. 
\end{equation}
It is well known that $\nu_{\psi}$ depends holomorphically on $\psi$. Now it follows from (\ref{3}) that 
\begin{equation*}
    \begin{split}
        |\mu_{\psi}(f_\phi(\zeta))|&\leqslant C_2 |\mathcal{S}(g_{\psi})(\gamma(f_\phi(\zeta)))|\rho_{D^*}^{-2}(\gamma(f_\phi(\zeta)))\\
        & = C_2 |\mathcal{S}(g_{\psi})(f_\phi(j(\zeta)))|\rho_{D^*}^{-2}(f_\phi(j(\zeta)))\\
        & = C_2 |\psi(j(\zeta))-\phi(j(\zeta))|\rho_{\mathbb{D}^*}^{-2}(j(\zeta))\\
        & = C_2 |\psi(\zeta^*)-\phi(\zeta^*)|\rho_{\mathbb{D}^*}^{-2}(\zeta^*)\\
    \end{split}
\end{equation*}
for every $\zeta \in \mathbb{D}$ with $\zeta^*=j(\zeta) \in \mathbb{D}^*$.

Since $\psi,\, \phi \in B_*^X(\mathbb{D}^*)$,
the above estimate implies that $\mu_{\psi}\circ f \in M_*^X(\mathbb{D})$.
Then, we see from (\ref{4}) that $\nu_{\psi} \in M_*^X(\mathbb{D})$. 
Since $\Phi(\nu_{\psi})=\psi$, we conclude that $\sigma: U_\phi \to M_{*}^{X}(\mathbb{D})$ 
defined by $\sigma(\psi)=\nu_\psi$
is a local holomorphic section to $\Phi$. This completes the proof. 
\end{proof}

\begin{corollary}\label{complexstructure}
The Bers embedding
$\beta: T_*^X \to B_*^X(\mathbb{D}^*)$ is a homeomorphism onto the domain 
$\beta(T) \cap B_*^X(\mathbb{D}^*)$ in $B_*^X(\mathbb{D}^*)$. Hence,
the Teich\-m\"ul\-ler space $T_*^X$ has the complex structure modeled on the complex Banach space $B_*^X(\mathbb{D}^*)$.
Under this complex structure, 
the projection $\pi:M_*^X(\mathbb{D}) \to T_*^X$ is also a holomorphic split submersion.
\end{corollary}

\begin{proof}
By the continuity of $\Phi:M_*^X(\mathbb{D}) \to B_*^X(\mathbb{D}^*)$, we see that $\beta$ is continuous.
For the other direction, the existence of the local continuous section to $\Phi$
shown in Theorem \ref{submersion} together with the continuity of the projection $\pi$
ensures the continuity of the inverse $\beta^{-1}:\beta(T) \cap B_*^X(\mathbb{D}^*) \to T_*^X$.
These facts prove that $\beta$ is a homeomorphism onto the image.
\end{proof}

Finally, we note that the corresponding result to Proposition \ref{L} is also valid for
the space of the holomorphic quadratic differentials.

\begin{proposition}\label{B}
For any $X = \{\xi_1,\dots, \xi_n\} \subset \mathbb{S}$, 
$B_*^X(\mathbb{D}^*) = B_*^{\xi_1}(\mathbb{D}^*)+\dots + B_{*}^{\xi_n}(\mathbb{D}^*)$.
\end{proposition}
\begin{proof}
For the Bers Schwarzian derivative map $\Phi: M_*^{X}(\mathbb{D}) \to B_*^X(\mathbb{D}^*)$, we consider 
its derivative $d_0\Phi: L_{*}^{X}(\mathbb{D})\to B_{*}^{X}(\mathbb{D}^*)$ at $0 \in M_*^{X}(\mathbb{D})$. By Proposition \ref{L}, $L_*^X(\mathbb{D}) = L_*^{\xi_1}(\mathbb{D})+\dots +L_*^{\xi_n}(\mathbb{D})$. Since $d_0\Phi$ is a linear map, we see that 
\begin{equation*}
    \begin{split}
      d_0\Phi(L_*^X(\mathbb{D})) & = d_0\Phi(L_*^{\xi_1}(\mathbb{D})+\dots +L_*^{\xi_n}(\mathbb{D})) \\
     & = d_0\Phi(L_*^{\xi_1}(\mathbb{D}))+\dots +d_0\Phi(L_*^{\xi_n}(\mathbb{D})) \\
     & = B_*^{\xi_1}(\mathbb{D}^*)+\dots+B_*^{\xi_n}(\mathbb{D}^*).\\
    \end{split}
\end{equation*}
Since $\Phi$ is a submersion by Theorem \ref{submersion}, $d_0\Phi: L_{*}^{X}(\mathbb{D})\to B_{*}^{X}(\mathbb{D}^*)$ is surjective, 
namely, $d_0\Phi(L_*^X(\mathbb{D}^*)) = B_*^X(\mathbb{D}^*)$. This completes the proof of Corollary \ref{B}.
\end{proof}


\end{document}